\documentclass[10pt]{amsart}

\usepackage{amssymb}
\usepackage{amsthm}
\usepackage{amsmath}

\usepackage[usenames]{color}
\definecolor{ANDREW}{RGB}{255,127,0}

\usepackage{tikz}
\usetikzlibrary{arrows,calc}

\usepackage[foot]{amsaddr}

\usepackage{hyperref}

\theoremstyle{plain}
\newtheorem{proposition}{Proposition}[section]
\newtheorem{theorem}[proposition]{Theorem}
\newtheorem{lemma}[proposition]{Lemma}
\newtheorem{corollary}[proposition]{Corollary}
\theoremstyle{definition}

\newtheorem{observation}[proposition]{Observation}
\theoremstyle{remark}
\newtheorem{remark}[proposition]{Remark}
\newtheorem{conjecture}[proposition]{Conjecture}

\DeclareMathOperator{\Aut}{Aut}

\DeclareMathOperator{\Vol}{Vol}

\DeclareMathOperator{\Euc}{Euc}

\DeclareMathOperator{\Cc}{\mathcal{C}}

\DeclareMathOperator{\Lc}{\mathcal{L}}

\DeclareMathOperator{\Oc}{\mathcal{O}}

\DeclareMathOperator{\Tc}{\mathcal{T}}

\DeclareMathOperator{\Bb}{\mathbb{B}}
\DeclareMathOperator{\Cb}{\mathbb{C}}

\DeclareMathOperator{\Nb}{\mathbb{N}}

\DeclareMathOperator{\Rb}{\mathbb{R}}

\DeclareMathOperator{\Zb}{\mathbb{Z}}

\newcommand{\abs}[1]{\left|#1\right|}

\newcommand{\norm}[1]{\left\|#1\right\|}
\newcommand{\wt}[1]{\widetilde{#1}}

\begin{document}

\title{Smoothly bounded domains covering finite volume manifolds}
\author{Andrew Zimmer}\address{Department of Mathematics, College of William and Mary.}
\email{amzimmer@wm.edu}
\date{\today}
\keywords{}
\subjclass[2010]{}

\begin{abstract} In this paper we prove: if a bounded domain with $C^2$ boundary  covers a manifold which has finite volume with respect to either the Bergman volume, the K{\"a}hler-Einstein volume, or the Kobayashi-Eisenman volume, then the domain is biholomorphic to the unit ball. This answers an old question of Yau. Further, when the domain is convex we can assume that the boundary only has $C^{1,\epsilon}$ regularity. 
\end{abstract}

\maketitle

\section{Introduction}

Given a domain $\Omega \subset \Cb^d$ let $\Aut(\Omega)$ denote the biholomorphism group of $\Omega$. When $\Omega$ is bounded, H. Cartan proved that $\Aut(\Omega)$ is a Lie group (with possibly infinitely many connected components) and acts properly on $\Omega$. 

An old theorem of Wong-Rosay~\cite{W1977, R1979} states that if $\Omega \subset \Cb^d$ is a bounded domain with $C^2$ boundary and $\Aut(\Omega)$ acts co-compactly on $\Omega$, then $\Omega$ is biholomorphic to the unit ball. According to Wong~\cite[p. 257]{W1977}, Yau suggested that the co-compactness condition could be replaced by the assumption that $\Omega$ covers a finite volume manifold. More precisely:

\begin{conjecture}[Yau] 
Let $\Omega \subset \Cb^d$ ($d \geq 2$) be a bounded pseudoconvex domain whose boundary is $C^2$. Assume that $\Omega$ has a (open) quotient of finite-volume (in the sense of K{\"a}hler-Einstein volume). Then $\Omega$ is biholomorphic to the unit ball in $\Cb^d$.
\end{conjecture}

Considering bounded domains that cover finite volume open manifolds seems more natural than studying those that cover compact manifolds. For instance, it is well known that $\Tc_{g}$, the Teichm{\"u}ller space of hyperbolic surfaces with genus $g$, is biholomorphic to a bounded domain and has a finite volume quotient. Further, Griffiths constructed the following examples.

\begin{theorem}\cite[Theorem I, Proposition 8.12]{G1971} Suppose $V$ is an irreducible, smooth, quasi-projective algebraic variety over the complex numbers. For any $x \in V$ there exists a Zariski neighborhood $U$ of $x$ such that $\wt{U}$, the universal cover of $U$, is biholomorphic to a bounded pseudoconvex domain in $\Cb^d$. Moreover, the Kobayashi-Eisenman volume of $U$ is finite. 
\end{theorem}

In this paper we answer Yau's question:

\begin{theorem}\label{thm:main} Suppose $\Omega \subset \Cb^d$ is a bounded pseudoconvex domain with $C^2$ boundary and $\Gamma \leq \Aut(\Omega)$ is a discrete group acting freely on $\Omega$. If $\Gamma \backslash \Omega$ has finite volume with respect to either the Bergman volume, the K{\"a}hler-Einstein volume, or the Kobayashi-Eisenman volume, then $\Omega$ is biholomorphic to the unit ball. 
\end{theorem}

\begin{remark} Recently, Liu and Wu~\cite{LW2018}, see Theorem~\ref{thm:LW} below, established the above theorem with the additional assumptions that \begin{enumerate}
\item $d=2$ and $\Omega$ is convex, or
\item $d>2$,  $\Omega$ is convex, and $\Gamma$ is irreducible. 
\end{enumerate}
\end{remark}

It is well known that Teichm{\"u}ller spaces admit a finite volume quotient and so Theorem~\ref{thm:main} provides a new proof of the following result. 

\begin{corollary}[{Yau~\cite[p. 328]{Y2011}}]\label{cor:teich}
Let $\Tc_{g}$ denote the Teichm{\"u}ller space of hyperbolic surfaces with genus $g$. If $g \geq 2$, then $\Tc_{g}$ is not biholomorphic to a bounded domain with $C^2$ boundary. 
\end{corollary}

\begin{remark} \ \begin{enumerate}
\item A theorem of Bers~\cite{B1960} says that $\Tc_{g}$ is biholomorphic to a bounded domain. 
\item Recently, Gupta and Seshadri~\cite[Theorem 1.2]{GS2017} provided a proof of Corollary~\ref{cor:teich} which relies on the ergodicity of the Teichm{\"u}ller geodesic flow.
\end{enumerate}
\end{remark}

If $\Omega \subset \Cb^d$ is a bounded domain, $\Gamma \leq \Aut(\Omega)$ is a discrete group acting freely on $\Omega$, and $\Gamma \backslash \Omega$ is a quasi-projective variety, then a result of Griffiths implies that $\Gamma \backslash \Omega$ has finite volume with respect to the Kobayashi-Eisenman volume (see Proposition 8.12 and the discussion following Question 8.13 in~\cite{G1971}). So we have the following corollary of Theorem~\ref{thm:main}. 

\begin{corollary}\label{cor:quasi_proj} Suppose $\Omega \subset \Cb^d$ is a bounded pseudoconvex domain with $C^2$ boundary and $\Gamma \leq \Aut(\Omega)$ is a discrete group acting freely on $\Omega$. If $\Gamma \backslash \Omega$ is a quasi-projective variety, then $\Omega$ is biholomorphic to the unit ball. 
\end{corollary}

The proof of Theorem~\ref{thm:main} uses the Levi form of the boundary and hence does not easily generalize to domains whose boundaries have less than $C^2$ regularity. However, by assuming our domain is convex we can lower the required regularity to $C^{1,\epsilon}$ for any $\epsilon > 0$. 

\begin{theorem}\label{thm:main_convex} Suppose $\Omega \subset \Cb^d$ is a bounded convex domain with $C^{1,\epsilon}$ boundary and $\Gamma \leq \Aut(\Omega)$ is a discrete group acting freely on $\Omega$. If $\Gamma \backslash \Omega$ has finite volume with respect to either the Bergman volume, the K{\"a}hler-Einstein volume, or the Kobayashi-Eisenman volume, then $\Omega$ is biholomorphic to the unit ball. 
\end{theorem}

\begin{remark} \ \begin{enumerate}
\item The proof will use a recent result of Liu and Wu~\cite{LW2018}, see Theorem~\ref{thm:LW} below, and a recent result in~\cite{Z2017convex}, see Theorem~\ref{thm:Z} below. 
\item It is conjectured that a bounded convex domain with a finite volume quotient (with no assumptions on the regularity of $\partial\Omega$) must be a bounded symmetric domain, see for instance~\cite[Conjecture 1.12]{LW2018}.
\item Using Theorem~\ref{thm:main_convex}, the hypothesis of Corollary~\ref{cor:quasi_proj} can be modified to assume that $\Omega$ is a bounded convex domain with $C^{1,\epsilon}$ boundary. Theorem~\ref{thm:main_convex} also can be used to show that $\Tc_g$ $(g \geq 2)$ is not biholomorphic to a convex domain with $C^{1,\epsilon}$ boundary, however a recent result of Markovic~\cite{M2017} implies that $\Tc_{g}$ is not biholomorphic any convex domain when $g \geq 2$ (with no regularity assumptions on the boundary of the convex domain). 
\end{enumerate}
\end{remark}

\subsection{Outline of the proofs} We will use a theorem of Wong and Rosay to prove Theorem~\ref{thm:main}.

\begin{theorem}[Wong-Rosay Ball Theorem~\cite{W1977, R1979}]\label{thm:WR} Suppose $\Omega \subset \Cb^d$ is a bounded domain. Assume that $\partial \Omega$ is $C^2$ and strongly pseudoconvex in a neighborhood of $\xi \in \partial \Omega$. If there exists some $z_0 \in \Omega$ and a sequence $\varphi_n \in \Aut(\Omega)$ such that $\varphi_n(z_0) \rightarrow \xi$, then $\Omega$ is biholomorphic to the unit ball. 
\end{theorem}

When $\Omega$ is a bounded domain with $C^2$ boundary, then there exists some $\xi \in \partial \Omega$ which is strongly pseudoconvex (see Observation~\ref{obs:strong_pc_exist} below). If $\Aut(\Omega)$ acts co-compactly on $\Omega$ then it is easy to show that there exists some $z_0 \in \Omega$ and a sequence $\varphi_n \in \Aut(\Omega)$ such that $\varphi_n(z_0) \rightarrow \xi$. So one has the following Corollary to Theorem~\ref{thm:WR}:

\begin{corollary} Suppose $\Omega \subset \Cb^d$ is a bounded domain with $C^2$ boundary. If $\Aut(\Omega)$ acts co-compactly on $\Omega$, then $\Omega$ is biholomorphic to the unit ball. 
\end{corollary}

In the case when $\Omega$ only admits a finite volume quotient, finding $z_0 \in \Omega$ and a sequence $\varphi_n \in \Aut(\Omega)$ such that $\varphi_n(z_0)$ converges to a certain boundary point $\xi \in \partial \Omega$ is much harder. We accomplish this task by considering the behavior of the Bergman distance and in particular the shape of horospheres near a strongly pseudoconvex point. The squeezing function also plays an important role in understanding the complex geometry of $\Omega$.

For convex domains, there are precise estimates for the Kobayashi distance and so in the proof of Theorem~\ref{thm:main_convex} we consider horospheres with respect to the Kobayashi distance (instead of the Bergman distance). Since the hypothesis of Theorem~\ref{thm:main_convex} only assumes $\partial \Omega$ has $C^{1,\epsilon}$ boundary there is no hope of using the Wong-Rosay Ball Theorem. Instead we reduce to two recent results about the automorphism group of convex domains. Before stating these results we need a few definitions: 
\begin{enumerate}
\item Given a bounded domain $\Omega \subset \Cb^d$ let $\Aut_0(\Omega)$ denote the connected component of the identity in $\Aut(\Omega)$. 
\item When $\Omega \subset \Cb^d$ is a bounded domain, the \emph{limit set of $\Omega$}, denoted $\Lc(\Omega)$ is the set of points $x \in \partial \Omega$ where there exists some $z \in \Omega$ and a sequence $\varphi_n \in \Aut(\Omega)$ such that $\varphi_n(z) \rightarrow x$.
\item Given a convex domain $\Omega \subset \Cb^d$ with $C^1$ boundary and $x \in \partial \Omega$, let $T_{x}^{\Cb} \partial \Omega \subset \Cb^d$ be the complex affine hyperplane tangent to $\partial\Omega$ at $x$. Then the \emph{closed complex face of $x$ in $\partial \Omega$} is the set $T_{x}^{\Cb} \partial \Omega \cap \partial \Omega$. 
\end{enumerate}

Liu and Wu recently proved the following rigidity result.

\begin{theorem}[{Liu-Wu~\cite{LW2018}\label{thm:LW}}] Suppose $\Omega \subset \Cb^d$ is a bounded convex domain, $\Gamma \leq \Aut(\Omega)$ is a discrete group acting freely on $\Omega$, and $\Gamma \backslash \Omega$ has finite volume with respect to either the Bergman volume, the K{\"a}hler-Einstein volume, or the Kobayashi-Eisenman volume. If either:
\begin{enumerate}
\item $\Gamma \leq \Aut_0(\Omega)$, 
\item $\Aut_0(\Omega) \neq 1$ and $\Gamma$ is irreducible,
\item $\Omega$ has $C^1$ boundary and $\Gamma$ is irreducible, 
\item $d=2$ and $\Aut_0(\Omega) \neq 1$, or
\item $d=2$ and $\Omega$ has $C^1$ boundary,
\end{enumerate}
then $\Omega$ is biholomorphic to a bounded symmetric domain.
\end{theorem}

\begin{remark} By the so-called rescaling method, part (3) (respectively part (5)) is a consequence of part (2) (respectively part (4)). Also, by a result of Mok and Tsai~\cite{MT1992}: if $\Omega$ is a bounded symmetric domain which is convex and has $C^1$ boundary, then $\Omega$ is biholomorphic to the unit ball. \end{remark}

We recently proved the following result.

\begin{theorem}\cite{Z2017convex}\label{thm:Z} Suppose $\Omega \subset \Cb^d$ is a bounded convex domain with $C^{1,\epsilon}$ boundary. If $\Lc(\Omega)$ intersects at least two different closed complex faces of $\partial \Omega$, then \begin{enumerate}
\item $\Aut(\Omega)$ has finitely many components,
\item there exists a compact normal subgroup $N \leq \Aut_0(\Omega)$ such that $\Aut_0(\Omega)/N$ is a non-compact simple Lie group with real rank one. 
\end{enumerate}
\end{theorem}

Hence to prove Theorem~\ref{thm:main_convex} it is enough to show that $\Lc(\Omega)$ intersects at least two different closed complex faces of $\partial \Omega$.

 \subsection*{Acknowledgements} This material is based upon work supported by the National Science Foundation under grant DMS-1760233.

\section{Preliminaries}

\subsection{Notations:}

Suppose $\Omega \subset \Cb^d$ is a bounded pseudoconvex domain, then
\begin{enumerate}
\item Let $k_\Omega: \Omega \times \Cb^d \rightarrow \Rb_{ \geq 0}$ denote the infinitesimal Kobayashi metric, $K_\Omega: \Omega \times \Omega \rightarrow \Rb_{\geq 0}$ denote the Kobayashi distance on $\Omega$, and $\Vol_{K}$ denote the Kobayashi-Eisenman volume form,
\item Let $g_B$ denote the Bergman metric on $\Omega$, $B_\Omega$ denote the Bergman distance on $\Omega$, and $\Vol_{B}$ denote the Riemannian volume form associated to $g_B$. We will also let $b_\Omega:\Omega \times \Cb^d \rightarrow \Rb$ denote the norm associated to $g_B$, that is 
\begin{align*}
b_\Omega(x;v) = \sqrt{g_B(v,v)}
\end{align*}
when $v \in T_x \Omega$.
\item Let $g_{KE}$ denote the K{\"a}hler-Einstein metric on $\Omega$ with Ricci curvature $-1$ constructed by Cheng-Yau~\cite{CY1980} when $\Omega$ has $C^2$ boundary and Mok-Yau~\cite{MY1983} in general. And let $\Vol_{KE}$ denote the Riemannian volume form associated to $g_{KE}$.
\end{enumerate}

Throughout the paper $\norm{ \cdot}$ will denote the standard Euclidean norm on $\Cb^d$. Given $z_0 \in \Cb^d$ and $r >0$ define
\begin{align*}
\Bb_d(z_0;r) = \{ z\in \Cb^d : \norm{z-z_0} < r\}.
\end{align*}
Finally, given a domain $\Omega \subset \Cb^d$ and $z \in \Omega$ define
\begin{align*}
\delta_\Omega(z) = \inf\{ \norm{w -z} : w \in \partial \Omega\}.
\end{align*}

\subsection{The squeezing function}\label{sec:squeeze}

Given a domain $\Omega \subset \Cb^d$ let $s_\Omega : \Omega \rightarrow (0,1]$ be the \emph{squeezing function on $\Omega$}, that is 
\begin{align*}
s_\Omega(z) = \sup\{ r : & \text{ there exists an one-to-one holomorphic map } \\
& f: \Omega \rightarrow \Bb_d(0;1) \text{ with } f(z)=0 \text{ and } \Bb_d(0;r) \subset f(\Omega) \}.
\end{align*}

In this section we recall a result of Sai-Kee Yeung.

\begin{theorem}\label{thm:squeeze}\cite[Theorem 2]{Y2009} Suppose $s > 0$ and $d > 0$. Then there exists $C, \delta, \epsilon, \kappa >0$ such that: if $\Omega \subset \Cb^d$, $z_0 \in \Omega$, $s_\Omega(z_0) > s$, and 
\begin{align*}
B_\epsilon = \{ z \in \Omega : B_\Omega(z_0,z) \leq \epsilon\},
\end{align*}
 then \begin{enumerate}
 \item $B_\epsilon \Subset \Omega$, 
\item $g_B$, $g_{KE}$, and $k_\Omega$ are all $C$-bi-Lipschitz on $B_\epsilon$,
\item the sectional curvature of $g_B$ is bounded in absolute value by $\kappa$ on $B_\epsilon$, 
\item the injectivity radius of $g_B$ is bounded below by $\delta$ on $B_\epsilon$, and
\item if $\Vol$ denotes either the Bergman volume, the K{\"a}hler-Einstein volume, or the Kobayashi-Eisenman volume, then
\begin{align*}
\Vol\Big( \{ z \in \Omega : B_\Omega(z, z_0) \leq r\}\Big) \geq r^{2d}/C
\end{align*}
for all $r \in [0,\epsilon]$.
\end{enumerate}
\end{theorem}

Parts (1)-(4) follow from~\cite[Theorem 2]{Y2009}. In~\cite[Theorem 2]{Y2009} it is assumed that $s_\Omega(z) > s$ for all $z \in \Omega$, however all the arguments are local in nature and can be easily modified to prove parts (1)-(4) in the above Theorem. Part (2) also follows from the proof of~\cite[Theorem 7.2]{LSY2004}.

Part (5) is a consequence of the definition and part (2): since $s_\Omega(z_0) > s$ we can assume that $z_0=0$ and 
\begin{align*}
\Bb_d(0;s) \subset \Omega \subset \Bb_d(0;1).
\end{align*}
Then 
\begin{align*}
k_{\Bb_d(0;1)} \leq k_\Omega \leq k_{\Bb_d(0;s)}
\end{align*}
on $\Bb_d(0;s)$. Then, from the well known explicit description of the Kobayashi metric on the ball and part (2), we see that there exists $C_1 > 0$ such that $g_B$, $g_{KE}$, and $k_\Omega$ are all $C_1$-bi-Lipschitz to the Euclidean metric on $\Bb_d(0;s/2)$. So we can find $C, \epsilon > 0$ such that: if $\Vol$ denotes either the Bergman volume or the K{\"a}hler-Einstein volume, then
\begin{align*}
\Vol\Big( \{ z \in \Omega : B_\Omega(z, z_0) \leq r\}\Big) \geq r^{2d}/C
\end{align*}
for all $r \in [0,\epsilon]$. Next let $\Vol_{K}$ denote the Kobayashi-Eisenman volume on $\Omega$ and $\Vol_{\Bb_d(0;1)}$ denote the Kobayashi-Eisenman volume on $\Bb_d(0;1)$. Then by definition 
\begin{align*}
\Vol_{K}(A) \geq \Vol_{\Bb_d(0;1)}(A)
\end{align*}
for all subsets $A\subset \Omega$. So from the well known explicit description of the Kobayashi-Eisenman volume for the ball, part (2), and by possibly modifying $C,\epsilon$ we can also assume that 
\begin{align*}
\Vol_K\Big( \{ z \in \Omega : B_\Omega(z, z_0) \leq r\}\Big) \geq r^{2d}/C
\end{align*}
for all $r \in [0,\epsilon]$.

\subsection{Invariant metrics near a strongly pseudoconvex point}

We will use the following well known facts about invariant metrics near a strongly pseudoconvex point. 

\begin{theorem}\label{thm:local} Suppose $\Omega \subset \Cb^d$ is a bounded domain. Assume that $\partial \Omega$ is $C^2$ and strongly pseudoconvex in a neighborhood of $\xi \in \partial \Omega$. Then there exists an neighborhood $U$ of $\xi$ in $\overline{\Omega}$ and $C > 0$ such that:
\begin{enumerate}
\item $k_\Omega$ and $g_B$ are $C$-bi-Lipshitz to each other on $U \cap \Omega$,
\item $k_\Omega(x;v) \geq C^{-1}\norm{v}\delta_\Omega(x)^{-1/2}$ for all $x \in U \cap \Omega$ and $v \in \Cb^d$, and
\item $g_B$ has negative sectional curvature on $U \cap \Omega$.

\end{enumerate}
\end{theorem}

\begin{proof} Fix open neighborhoods $V_2 \Subset V_1$ of $\xi$ such that there exist a holomorphic embedding $\varphi: V_1 \rightarrow \Cb^d$ with $\varphi(V_2 \cap \Omega)$ a convex domain which is strongly convex near $\varphi(\xi)$.

By~\cite[Theorem 2.1]{FR1987} there exists a neighborhood $V_3$ of $\xi$ such that $V_3 \Subset V_2$ and
\begin{align*}
k_\Omega(x;v) \leq k_{V_2 \cap \Omega}(x;v) \leq 2 k_\Omega(x;v)
\end{align*}
for all $x \in V_3$ and $v \in \Cb^d$ (notice that the first inequality is by definition). Further, by~\cite[Theorem 1]{DFH1984}  there exists $C_0 > 1$ such that
\begin{align*}
\frac{1}{C_0} b_{\Omega \cap V_2}(x;v) \leq b_\Omega(x;v) \leq C_0b_{\Omega \cap V_2}(x;v)
\end{align*}
for all $x \in V_3$ and $v \in \Cb^d$. 

Now since $V_2 \cap \Omega$ is biholomorphic to a convex domain, a result of Frankel~\cite{F1991} implies that $b_{\Omega \cap V_2}$ and $k_{\Omega \cap V_2}$ are $C_1$-bi-Lipschitz to each other for some $C_1 >1$. So we see that $k_\Omega$ and $g_B$ are $C$-bi-Lipshitz to each other on $V_3 \cap \Omega$ for some $C>1$. 

Given a domain $\Oc \subset \Cb^d$, $x \in \Oc$, and nonzero $v \in \Cb^d$ define
\begin{align*}
\delta_{\Oc}(x;v) = \inf\{ \norm{y-x} : y \in \partial \Omega \cap (x+\Cb v)\}.
\end{align*}
Since $\Cc=\varphi(V_2 \cap \Omega)$ is convex, a result of Graham~\cite{G1990,G1991} says that 
\begin{align*}
\frac{\norm{v}}{2\delta_{\Cc}(x;v)} \leq k_{\Cc}(x;v) \leq \frac{\norm{v}}{\delta_{\Cc}(x;v)}
\end{align*}
for all $x \in \Cc$ and $v \in \Cb^d$. Then, since $\Cc$ is strongly convex at $\varphi(\xi)$, there exists a neighborhood $W$ of $\varphi(\xi)$ and some $C_2 > 0$ such that 
\begin{align*}
C_2 \frac{\norm{v}}{\delta_{\Cc}(x)^{1/2}} \leq k_{\Cc}(x;v)
\end{align*}
for all $x \in W$ and $v \in \Cb^d$. Since $V_2 \Subset V_1$, the map $\varphi:V_1 \rightarrow \Cb^d$ is bi-Lipschitz on $V_2$, so  by possibly shrinking $V_3$ and increasing $C$ we can assume that 
\begin{align*}
\frac{1}{C} \frac{\norm{v}}{\delta_{\Omega}(x)^{1/2}} \leq k_{\Omega}(x;v)
\end{align*}
for all $x \in V_3$ and $v \in \Cb^d$. 

Finally, part (3) follows  from~\cite[Theorem 1]{KY1996}.

\end{proof}

\subsection{Completeness of the Bergman metric}

We will use the following  fact about the Bergman metric:

\begin{theorem}[Ohsawa~\cite{O1981}] If $\Omega \subset \Cb^d$ is a bounded pseudoconvex domain with $C^1$ boundary, then the Bergman metric is a complete Riemannian metric on $\Omega$.
\end{theorem}

\begin{remark} It is also known that the Bergman metric is complete on the more general class of hyperconvex domains, see~\cite{H1999} and~\cite{BP1998b}.\end{remark}

\subsection{A local version of E. Cartan's fixed point theorem}

E. Cartan showed that a compact group $G$ acting by isometries on $(X,g)$ a complete simply connected Riemannian manifold with non-positive sectional curvature always has a fixed point. One proof, see for instance~\cite[p. 21]{E1996}, uses the following lemma: if $K \subset X$ is compact, then the function
\begin{align*}
f(x) = \sup\{ d(x,k) : k \in K\}
\end{align*}
has a unique minimum in $X$. In this section we observe a local version of this lemma which will allow us to show that a certain compact subgroup has a fixed point in the proof of Theorem~\ref{thm:main}.

Given a complete Riemannian manifold $(X,g)$, $x_0 \in X$, and $R >0$ let $B_{(X,g)}(x_0,R)$ denote the open metric ball of radius $R$ centered at $x_0$. 

\begin{proposition}\label{prop:COM} Suppose $(X,g)$ is a complete Riemannian manifold, $x_0 \in X$, $R>0$, the metric $g$ has non-positive sectional curvature on $B_{(X,g)}(x_0,8R)$, and $g$ has injectivity radius at least $16R$ at each point in $B_{(X,g)}(x_0,8R)$. If $K \subset B_{(X,g)}(x_0,R)$ is compact, then the function
\begin{align*}
f(x) = \sup\{ d(x,k) : k \in K\}
\end{align*}
has a unique minimum in $X$.
\end{proposition}

The following proof is nearly identical to the proof of the Lemma on p. 21 in~\cite{E1996}, but we provide the details for the reader's convenience. 

\begin{proof}
When $\alpha \in [0,4]$, every two points in $B_{(X,g)}(x_0,\alpha R)$ are joined by a unique geodesic and this geodesic is contained in $B_{(X,g)}(x_0,2 \alpha R)$. 

Since $g$ is non-positively curved on $B_{(X,g)}(x_0, 8R)$ and has injectivity radius at least $16R$ at each point in $B_{(X,g)}(x_0,8R)$ the Rauch comparison theorem implies (see \cite[p. 73]{H2001}): if $\Tc$ is a geodesic triangle contained in $B_{(X,g)}(x_0,4R)$ with side lengths $a,b,c$ then
\begin{align}
\label{eq:law_of_cosines}
a^2 + b^2 -2ab \cos \theta \leq c^2
\end{align}
where $\theta$ is the angle at the vertex opposite to the side of length $c$. 

Since $f$ is a proper continuous function there exists at least one minimum. Since $f(x) > R$ when $x \in X \setminus B_{(X,g)}(x_0,2R)$ and $f(x_0) \leq R$ any minimum of $f$ is in $B_{(X,g)}(x_0,2R)$. 

Suppose for a contradiction that there exists two distinct minimum points $x,y$ of $f$. Let $\sigma:[0,T] \rightarrow X$ denote the unique geodesic with $\sigma(0)=x$ and $\sigma(T)=y$. Let $m = \sigma(T/2)$. Then consider some $k \in K$ and let $\gamma: [0,S] \rightarrow X$ denote the unique geodesic in $X$ with $\gamma(0)=m$ and $\gamma(S) = k$. Since 
\begin{align*}
\angle_m(-\sigma^\prime(T/2), \gamma^\prime(0)) + \angle_m(\sigma^\prime(T/2), \gamma^\prime(0))=\pi,
\end{align*}
by relabelling $x,y$ we can assume that $\theta:=\angle_m(\sigma^\prime(T/2), \gamma^\prime(0)) \geq \pi/2$. Then Equation~\eqref{eq:law_of_cosines} implies that 
\begin{align*}
d(y,k)^2 \geq d(y,m)^2 + d(m,k)^2 - 2d(y,m)d(m,k)\cos \theta > d(m,k)^2.
\end{align*}
So $d(m,k) < d(y,k) \leq f(y)$. Since $k \in K$ was arbitrary and $K$ is compact, we then have $f(m) < f(y)$ which is a contradiction. 
\end{proof}

\section{An estimate for the Bergman distance} 

\begin{theorem}\label{thm:GP} Suppose $\Omega \subset \Cb^d$ is a bounded pseudoconvex domain. Assume that $\partial \Omega$ is $C^2$ and strongly pseudoconvex in a neighborhood of $\xi \in \partial \Omega$. If $z_0 \in \Omega$ and $\epsilon_0 > 0$, then there exists $\epsilon \in (0,\epsilon_0)$ and $R >0$ such that 
\begin{align*}
B_\Omega(z,w) \geq B_\Omega(z,z_0)+B_\Omega(z_0,w) -R
\end{align*}
for all $z,w \in \Omega$ with $\norm{z-\xi} < \epsilon$ and $\norm{w-\xi} > \epsilon$. 
\end{theorem}

\begin{remark} This says that a point $z$ near $\xi$ and point $w$ far away from $\xi$ can be joined by a path that passes through $z_0$ and is length minimizing up to an error of $R$. \end{remark}

The following argument is based on the proof of ~\cite[Lemma 36]{K2005b} which establishes a similar estimate for the Kobayashi distance.

\begin{proof} By Theorem~\ref{thm:local} there exists an neighborhood $U$ of $\xi$ and some $C > 1$ such that 
\begin{align*}
\frac{1}{C} k_\Omega(x;v) \leq b_\Omega(x;v) \leq C k_\Omega(x;v)
\end{align*}
and
\begin{align*}
\frac{1}{C} \frac{\norm{v}}{\delta_\Omega(x)^{1/2}} \leq b_\Omega(x;v)
\end{align*}
for all $x \in \Omega \cap U$ and $v \in  \Cb^d$. 

By definition
\begin{align*}
k_\Omega(x;v) \leq \frac{\norm{v}}{\delta_\Omega(x)}
\end{align*}
and so
\begin{align*}
b_\Omega(x;v) \leq C\frac{\norm{v}}{\delta_\Omega(x)}
\end{align*}
for $x \in U \cap \Omega$ and $v \in \Cb^d$. Then since $\partial \Omega$ is $C^2$ near $\xi$ one can consider parametrizations of inward pointing normal lines to show that there exists $\alpha, \beta > 0$ and a neighborhood $V \subset U$ of $\xi$ such that
\begin{align*}
B_\Omega(z_0, z) \leq \alpha + \beta \log \frac{1}{\delta_\Omega(z)}
\end{align*}
for all $z \in V \cap \Omega$. 

Now fix $\epsilon \in (0,\epsilon_0)$ such that 
\begin{align*}
\{ z \in \Cb^d : \norm{z-\xi} < 2\epsilon \} \subset V.
\end{align*}
Consider points $z,w \in \Omega$ with $\norm{z-\xi} < \epsilon$ and $\norm{w-\xi} > \epsilon$. Let $\sigma: [0,T] \rightarrow \Omega$ be a geodesic (with respect to the Bergman distance) joining $z$ and $w$. Define 
\begin{align*}
T_0 = \max\left\{ t \in [0,T] : \sigma([0,t]) \subset \overline{\Bb_d(z; \epsilon)} \right\}.
\end{align*}
Then let $\tau \in [0,T_0]$ be such that 
\begin{align*}
\delta_\Omega(\sigma(\tau)) = \max \{ \delta_\Omega(\sigma(t)) : t \in [0,T_0]\}.
\end{align*}
Now for $t \in [0,T_0]$ we have 
\begin{align*}
\abs{t-\tau} &= B_\Omega(\sigma(t), \sigma(\tau)) \leq B_\Omega(\sigma(t), z_0)+B_\Omega(z_0, \sigma(\tau))\\
& \leq 2\alpha + \beta \log \frac{1}{\delta_\Omega(\sigma(t))\delta_\Omega(\sigma(\tau))}
\end{align*}
So 
\begin{align*}
\delta_\Omega(\sigma(t)) \leq \sqrt{\delta_\Omega(\sigma(t))\delta_\Omega(\sigma(\tau))} \leq 
\exp \left( \frac{ -\abs{t-\tau}+2\alpha}{2\beta} \right). 
\end{align*}

Now fix $M > 0$ such that 
\begin{align*}
\int_{M}^{\infty} \exp \left( \frac{ -r+2\alpha}{4\beta} \right) dr < \epsilon/(4C).
\end{align*}
Then
\begin{align*}
\epsilon = & \norm{\sigma(0)-\sigma(T_0)} \leq \int_0^{T_0} \norm{\sigma^\prime(t)} dt \leq C \int_0^{T_0} \delta_\Omega(\sigma(t))^{1/2} dt 
\end{align*}
since $b_\Omega(\sigma(t); \sigma^\prime(t))=1$. Then 
\begin{align*}
\epsilon & \leq C \int_{[0,T_0] \cap (\tau-M, \tau +M)} \delta_\Omega(\sigma(t))^{1/2} dt + C \int_{[0,T_0] \cap (\tau-M, \tau +M)^c} \delta_\Omega(\sigma(t))^{1/2} dt \\
& \leq 2CM \delta_\Omega(\sigma(\tau))^{1/2} +  2C\int_{M}^\infty  \exp \left( \frac{ -r+2\alpha}{4\beta} \right) dr \\
& \leq 2CM \delta_\Omega(\sigma(\tau))^{1/2} +  \epsilon/2.
\end{align*}
So 
\begin{align*}
\delta_\Omega(\sigma(\tau))^{1/2} \geq \epsilon/(4CM).
\end{align*}
Then 
\begin{align*}
B_\Omega(z,w) &=B_\Omega(z,\sigma(\tau))+B_\Omega(\sigma(\tau),w) \geq B_\Omega(z,z_0)+B_\Omega(z_0,w) - 2B_\Omega(z_0,\sigma(\tau)) \\
& \geq B_\Omega(z,z_0)+B_\Omega(z_0,w) - R
\end{align*}
where
\begin{align*}
R = 2\alpha + 4\beta \log \frac{4CM}{\epsilon}.
\end{align*}
Notice that $R$ does not depend on $z$ or $w$, so the proof is complete. 
\end{proof}

\section{Proof of Theorem~\ref{thm:main}}

For the rest of the section suppose that $\Omega \subset \Cb^d$ is a bounded pseudoconvex domain with $C^{2}$ boundary and $\Gamma \leq \Aut(\Omega)$ is a discrete group acting freely on $\Omega$. Further assume that  $\Vol(\Gamma \backslash \Omega) < +\infty$ where $\Vol$ is either the Bergman volume, the K{\"a}hler-Einstein volume, or the Kobayashi-Eisenman volume.

By replacing $\Omega$ with an affine translate we can assume that $0 \in \Omega$, $\Omega \subset \Bb_d(0;1)$, and $(1,0,\dots,0) \in \partial \Omega$. Then, since $\partial \Omega$ is $C^2$, there exists some $r\in (0,1)$ such that 
\begin{align*}
\Bb_d( (r,0,\dots,0); 1-r) \subset \Omega \subset \Bb_d(0;1).
\end{align*}

\begin{observation}\label{obs:strong_pc_exist} $\xi=(1,0,\dots,0)$ is a strongly pseudoconvex point of $\partial \Omega$. 
\end{observation}

\begin{proof}
This is a simple consequence of the fact that $(1,0,\dots, 0) \in \partial \Omega$ and $\Omega \subset \Bb_d(0;1)$, see for instance~\cite[Lemma 4.1]{GS2017}.
\end{proof}

\begin{observation} Let $w_t = (t,0,\dots,0) \in \Cb^d$. Then there exists some $s_0 > 0$ such that $s_\Omega(w_t) \geq s_0$ for $t \in [r,1)$.
\end{observation}

\begin{proof}

For $t \in [r,1)$ consider the transformation 
\begin{align*}
\varphi(z_1,\dots, z_d) = \left( \frac{z_1-t}{tz_1-1}, \frac{(1-t^2)^{1/2}}{tz_1-1}z_2, \dots, \frac{(1-t^2)^{1/2}}{tz_1-1}z_d \right). 
\end{align*}
Then $\varphi \in \Aut(\Bb_d(0;1))$ and $\varphi(0)=w_t$. We claim that 
\begin{align*}
\varphi( \Bb_d(0;s_0)) \subset \Omega
\end{align*}
 where
 \begin{align*}
 s_0 = \frac{1-r}{12\sqrt{d}}.
 \end{align*}

Suppose $z \in \Bb_d(0;s_0)$. Then $\norm{z} \leq 1/2$ and so 
\begin{align*}
\abs{tz_1-1} \geq 1/2.
\end{align*}
Then
\begin{align*}
\abs{\frac{z_1-t}{tz_1-1}-r}^2 
& = \abs{ (t-r)+\frac{(1-t^2)z_1}{tz_1-1}}^2 \leq (t-r)^2 + \frac{2(t-r)(1-t^2)}{\abs{tz_1-1}} \abs{z_1} + \frac{(1-t^2)^2}{\abs{tz_1-1}^2} \abs{z_1}^2 \\
& \leq (t-r)^2  + 4(1-t)\abs{z_1}+ 16 (1-t)\abs{z_1}^2 \\
& \leq (t-r)^2 + 4 (1-t)\abs{z_1} + 8(1-t)\abs{z_1}\\
& \leq (t-r)^2 + 12(1-t) \abs{z_1}.
\end{align*}
We also have
\begin{align*}
(t-r)^2 -(1-r)^2 = (2r-1-t)(1-t) \leq (r-1)(1-t)
\end{align*}
and
\begin{align*}
\abs{\frac{(1-t^2)^{1/2}}{tz_1-1}z_i}^2 \leq 8(1-t)\abs{z_i}^2 \leq 4(1-t)\abs{z_i}.
\end{align*}
So
\begin{align*}
\norm{\varphi(z)-w_r}^2
& \leq (1-r)^2 +(r-1)(1-t) + 12(1-t)(\abs{z_1}+\dots+\abs{z_d})\\
& \leq (1-r)^2 +(r-1)(1-t) + 12\sqrt{d}(1-t) \norm{z} \\
& < (1-r)^2.
\end{align*}
So $\varphi(z) \in \Bb_d(w_r;1-r) \subset \Omega$. Since $z \in \Bb_d(0;s_0)$ was arbitrary, we then have
\begin{align*}
\varphi(\Bb_d(0;s_0)) \subset \Omega.
\end{align*}

Then $\varphi^{-1}(w_t) = 0$ and 
\begin{align*}
\Bb_d(0;s_0) \subset \varphi^{-1}(\Omega) \subset \Bb_d(0;1),
\end{align*}
so $s_\Omega(w_t) \geq s_0$.
\end{proof}

Then fix a sequence $r_n \nearrow 1$ and consider the points $y_n=(r_n,0,\dots,0) \in \Omega$. For each $n \in \Nb$ define
\begin{align*}
\delta_n = \min_{\gamma \in \Gamma \setminus \{1\}} B_\Omega(y_n, \gamma y_n).
\end{align*}
Then the quotient map $\pi : \Omega \rightarrow \Gamma \backslash \Omega$ restricts to an embedding on
\begin{align*}
B_n = \{ z \in \Omega : B_\Omega(z, y_n) < \delta_n/2\}.
\end{align*}
Further, by Theorem~\ref{thm:squeeze} there exists some $C,\epsilon_0 > 0$ such that
\begin{align*}
{ \rm Vol}(\pi(B_n)) \geq C\min\{\epsilon_0^{2d}, \delta_n^{2d}\}.
\end{align*}

After passing to a subsequence we can assume that 
\begin{align*}
\lim_{n \rightarrow \infty} \delta_n = \delta \in \Rb_{\geq 0} \cup \{ \infty\}.
\end{align*}

\noindent \textbf{Case 1:} $\delta \neq 0$. Since $\Vol(\Gamma \backslash \Omega) < \infty$, the set $\{ \pi(y_n) : n \in \Nb\}$ must be relatively compact in $\Gamma \backslash \Omega$. So for each $n$, there exist some $\gamma_n \in \Gamma$ such that the set $\{ \gamma_n y_n : n \in \Nb\}$ is relatively compact in $\Omega$. Then we can pass to a subsequence such that $\gamma_n y_n \rightarrow y \in \Omega$. Then $\gamma_n^{-1}y \rightarrow \xi$. So $\Omega$ is biholomorphic to the ball by Theorem~\ref{thm:WR}. \\

\noindent \textbf{Case 2:} $\delta=0$. Pick $\gamma_n \in \Gamma$ such that 
\begin{align*}
B_\Omega(\gamma_n y_n, y_n) = \delta_n.
\end{align*}

\noindent \textbf{Case 2(a):} The set $\{ \gamma_1, \gamma_2, \dots\}$ is infinite. Since $\Gamma$ is discrete, by passing to a subsequence we can suppose that $\gamma_n \rightarrow \infty$ in $\Aut(\Omega)$. Fix some $z_0 \in \Omega$. By passing to another subsequence we can assume that $\gamma_n^{-1} z_0 \rightarrow \eta \in \partial \Omega$.  Since $(\Omega, B_\Omega)$ is a complete proper metric space we must have 
\begin{align*}
B_\Omega(z_0, \gamma_n^{-1} z_0) \rightarrow \infty.
\end{align*}
We claim that $\eta= \xi$. Suppose not, then by Theorem~\ref{thm:GP} there exists $R > 0$ such that 
\begin{align*}
B_\Omega(\gamma_n^{-1}z_0, z_0) &+ B_\Omega(z_0, y_n) - B_\Omega(\gamma_n^{-1} z_0, y_n) \leq R.
\end{align*}
However
\begin{align*}
B_\Omega(\gamma_n^{-1}z_0, z_0) &+ B_\Omega(z_0, y_n) - B_\Omega(\gamma_n^{-1} z_0, y_n) \\
& = B_\Omega(\gamma_n^{-1}z_0, z_0) + B_\Omega(z_0, y_n) - B_\Omega(z_0, \gamma_n y_n) \\
& \geq B_\Omega(\gamma_n^{-1}z_0, z_0) - B_\Omega(\gamma_n y_n, y_n) \rightarrow \infty.
\end{align*}
So we have a contradiction and hence $\xi = \eta$. So $\Omega$ is biholomorphic to the unit ball by Theorem~\ref{thm:WR}.\\

\noindent \textbf{Case 2(b):} The set $\{ \gamma_1, \gamma_2, \dots\}$ is finite. By passing to a subsequence we can suppose that $\gamma_n = \gamma$ in for all $n \in \Nb$. Fix some $z_0 \in \Omega$ and consider the functions 
\begin{align*}
b_n(z) = B_\Omega(z,y_n)-B_\Omega(y_n,z_0).
\end{align*}
Since $b_n(z_0)=0$ and each $b_n$ is 1-Lipschitz (with respect to the Bergman distance) we can pass to a subsequence such that $b_n \rightarrow b$ locally uniformly. Then
\begin{align*}
b(\gamma^{-1}z) = \lim_{n \rightarrow \infty} B_\Omega(\gamma^{-1} z,y_n)-B_\Omega(y_n,z_0) =  \lim_{n \rightarrow \infty} B_\Omega(z,\gamma y_n)-B_\Omega(y_n,z_0) =b(z)
\end{align*}
since
\begin{align*}
\abs{ B_\Omega(z,\gamma y_n)-B_\Omega(z, y_n)} \leq B_\Omega(y_n, \gamma y_n) \rightarrow 0.
\end{align*}
So 
\begin{align*}
b(\gamma^{-n} z_0) = b(z_0)=0
\end{align*}
for all $n \in \Nb$. 

\begin{observation} For any $t \in \Rb$
\begin{align*}
\overline{b^{-1}\Big((-\infty, t]\Big)}^{\Euc} \cap \partial \Omega = \{ \xi \}.
\end{align*}
\end{observation}

\begin{proof} Suppose $w_m \in b^{-1}\Big((-\infty, t]\Big)$ and $w_m \rightarrow \eta \in \partial \Omega$. If $\eta \neq \xi$, then Theorem~\ref{thm:GP} implies that there exists $R >0$ such that 
\begin{align*}
B_\Omega(w_m, z_0) &+ B_\Omega(z_0, y_n) - B_\Omega(w_m, y_n) \leq R.
\end{align*}
Then 
\begin{align*}
b(w_m) = \lim_{n \rightarrow \infty} B_\Omega(w_m, y_n)-B_\Omega(z_0, y_n) \geq B_\Omega(w_m, z_0)-R.
\end{align*}
However $B_\Omega(w_m, z_0) \rightarrow \infty$ since $B_\Omega$ is a proper metric on $\Omega$. So we have a contradiction.
\end{proof}

Using the previous observation, if $\gamma^{-n} z_0$ is unbounded in $\Omega$, then there exists $n_k \rightarrow \infty$ such that $\gamma^{-n_k} z_0 \rightarrow \xi$. Hence, in this case, $\Omega$ is biholomorphic to the unit ball by Theorem~\ref{thm:WR}.

It remains to consider the case where the sequence $\gamma^{-n} z_0$ is bounded in $\Omega$. Since $\Gamma$ is discrete and acts properly on $\Omega$, in this case
\begin{align*}
M:={\rm order}(\gamma) < \infty.
\end{align*}
We claim that $\gamma$ has a fixed point in $\Omega$. First, notice that
\begin{align*}
K_\Omega(\gamma^m y_n, y_n) \leq (M-1)\delta_n
\end{align*}
for all $m \in \Zb$.  By Theorem~\ref{thm:squeeze} there exists some $\tau> 0$ such that the injectivity radius of $g_\Omega$ is bounded below by $\tau$ on each $U_n = \{ z \in \Omega : B_\Omega(z_0,y_n) \leq \tau\}$. By Theorem~\ref{thm:local}, $g_B$ is negatively curved on $U_n$ when $n$ is large. Then since $\delta_n \rightarrow 0$, Proposition~\ref{prop:COM} implies that when $n$ is large the function 
\begin{align*}
f_n(x) = \sup\{ B_\Omega(\gamma^m y_n, x) : m=0,1,\dots, M-1\}
\end{align*}
has a unique minimum $c_n$ in $\Omega$. Since 
\begin{align*}
\gamma \left\{ y_n, \gamma y_n, \gamma^2 y_n, \dots, \gamma^{M-1} y_n\right\} = \left\{ y_n, \gamma y_n, \gamma^2 y_n, \dots, \gamma^{M-1} y_n\right\},
\end{align*}
we then have $\gamma c_n = c_n$. So $\gamma$ has a fixed point in $\Omega$. Since $\Gamma$ acts freely on $\Omega$, we have a contradiction. 

\section{The convex case}

Before starting the proof of Theorem~\ref{thm:main_convex} we will recall some results about convex domains. 

As in Section~\ref{sec:squeeze}, let $s_\Omega: \Omega \rightarrow (0,1]$ denote the squeezing function on a bounded domain $\Omega \subset \Cb^d$.

\begin{theorem}\cite{F1991, KZ2016, NA2017} For any $d>0$ there exists some $s=s(d) > 0$ such that: if $\Omega \subset \Cb^d$ is a bounded convex domain, then $s_\Omega(z) \geq s$ for all $z \in \Omega$. 
\end{theorem}

We will also need the following facts about the Kobayashi distance. 

\begin{proposition}\label{prop:convex_complete} Suppose $\Omega \subset \Cb^d$ is a bounded convex domain. Then the metric space $(\Omega, K_\Omega)$ is proper and Cauchy complete. 
\end{proposition}

For a proof of  Proposition~\ref{prop:convex_complete} see for instance~\cite[Proposition 2.3.45]{A1989}.

\begin{theorem}\label{thm:GP_convex}\cite[Theorem 4.1]{Z2017} Suppose $\Omega \subset \Cb^d$ is a bounded convex domain with $C^{1,\epsilon}$ boundary. If $\xi, \eta \in \partial \Omega$ and $T_{\xi}^{\Cb} \partial \Omega \neq T_\eta^{\Cb} \partial \Omega$, then 
\begin{align*}
\limsup_{x \rightarrow \xi, y \rightarrow \eta} K_\Omega(x, z_0) + K_\Omega(z_0, y) - K_\Omega(x, y) < \infty
\end{align*}
for some (hence any) $z_0 \in \Omega$.
\end{theorem}

\begin{remark} This says that a point $x$ near $\xi$ and point $y$ near $\eta$ can be joined by a path that passes through $z_0$ and is length minimizing up to a bounded error. \end{remark}

\subsection{Proof of Theorem~\ref{thm:main_convex}}

For the rest of the section suppose that $\Omega \subset \Cb^d$ is a bounded convex domain with $C^{1,\epsilon}$ boundary and $\Gamma \leq \Aut(\Omega)$ is a discrete group acting freely on $\Omega$. Further assume that  $\Vol(\Gamma \backslash \Omega) < +\infty$ where $\Vol$ is either the Bergman volume, the K{\"a}hler-Einstein volume, or the Kobayashi-Eisenman volume.

Using Theorem~\ref{thm:LW} and Theorem~\ref{thm:Z} it is enough to show that $\Lc(\Omega)$ intersects at least two different closed complex faces of $\partial \Omega$. 
\begin{lemma} If $\xi \in \partial \Omega$, then $\Lc(\Omega) \cap T_{\xi}^{\Cb} \partial \Omega \neq \emptyset$. \end{lemma}

The proof of the Lemma is nearly identical to the proof of Theorem~\ref{thm:main}, but we provide the complete argument for the reader's convenience. 

\begin{proof} By replacing $\Omega$ with an affine translate, we may assume that $\xi = (1,0,\dots,0)$ and $0 \in \Omega$. Then fix a sequence $r_n \nearrow 1$ and consider the points $y_n=(r_n,0,\dots,0) \in \Omega$. For each $n \in \Nb$ define
\begin{align*}
\delta_n = \min_{\gamma \in \Gamma \setminus\{1\}} K_\Omega(y_n, \gamma y_n).
\end{align*}
Now for each $n \in \Nb$ the quotient map $\pi : \Omega \rightarrow \Gamma \backslash \Omega$ restricts to an embedding on
\begin{align*}
B_n = \{ z \in \Omega : K_\Omega(z, y_n) < \delta_n/2\}.
\end{align*}
Further, by Theorem~\ref{thm:squeeze} there exists some $C,\epsilon_0 > 0$ such that
\begin{align*}
{ \rm Vol}(\pi(B_n)) \geq C\min\{\epsilon_0^{2d}, \delta_n^{2d}\}.
\end{align*}

After passing to a subsequence we can assume that 
\begin{align*}
\lim_{n \rightarrow \infty} \delta_n = \delta \in \Rb_{\geq 0} \cup \{ \infty\}.
\end{align*}

\noindent \textbf{Case 1:} $\delta \neq 0$. Since $\Vol(\Gamma \backslash \Omega) < \infty$, the set $\{ \pi(y_n) : n \in \Nb\}$ must be relatively compact in $\Gamma \backslash \Omega$. So for each $n$, there exist some $\gamma_n \in \Gamma$ such that the set $\{ \gamma_n y_n : n \in \Nb\}$ is relatively compact in $\Omega$. Then we can pass to a subsequence such that $\gamma_n y_n \rightarrow y \in \Omega$. Then $\gamma_n^{-1}y \rightarrow \xi$. So $\xi \in \Lc(\Omega)$. \\

\noindent \textbf{Case 2:} $\delta=0$. Then pick $\gamma_n \in \Gamma$ such that 
\begin{align*}
K_\Omega(\gamma_n y_n, y_n) = \delta_n.
\end{align*}

\noindent \textbf{Case 2(a):} The set $\{ \gamma_1, \gamma_2, \dots\}$ is infinite. Since $\Gamma$ is discrete, by passing to a subsequence we can suppose that $\gamma_n \rightarrow \infty$ in $\Aut(\Omega)$. Fix some $z_0 \in \Omega$. By passing to another subsequence we can assume that $\gamma_n^{-1} z_0 \rightarrow \eta \in \partial \Omega$.  Since $(\Omega, K_\Omega)$ is a complete proper metric space we must have 
\begin{align*}
K_\Omega(z_0, \gamma_n^{-1} z_0) \rightarrow \infty.
\end{align*}
We claim that $\eta \in T_\xi^{\Cb} \partial \Omega$. Suppose not, then by Theorem~\ref{thm:GP_convex} there exists $R > 0$ such that 
\begin{align*}
K_\Omega(\gamma_n^{-1}z_0, z_0) &+ K_\Omega(z_0, y_n) - K_\Omega(\gamma_n^{-1} z_0, y_n) \leq R.
\end{align*}
However
\begin{align*}
K_\Omega(\gamma_n^{-1}z_0, z_0) &+ K_\Omega(z_0, y_n) - K_\Omega(\gamma_n^{-1} z_0, y_n) \\
& = K_\Omega(\gamma_n^{-1}z_0, z_0) + K_\Omega(z_0, y_n) - K_\Omega(z_0, \gamma_n y_n) \\
& \geq K_\Omega(\gamma_n^{-1}z_0, z_0) - K_\Omega(\gamma_n y_n, y_n) \rightarrow \infty.
\end{align*}
So we have a contradiction and hence $\eta \in T_\xi^{\Cb} \partial \Omega$. \\

\noindent \textbf{Case 2(b):} The set $\{ \gamma_1, \gamma_2, \dots\}$ is finite. By passing to a subsequence we can suppose that $\gamma_n = \gamma$ in for all $n \in \Nb$. 

Fix some $z_0 \in \Omega$. If the set $\{ \gamma^n(z_0) : n \in \Nb\}$ is relatively compact in $\Omega$, then $\gamma$ has a fixed point in $\Omega$ (see for instance~\cite[Theorem 5.1]{Z2017}). So, since $\Gamma$ acts freely on $\Omega$, the set  $\{ \gamma^n(z_0) : n \in \Nb\}$ must be unbounded in $\Omega$. 

Next consider the functions 
\begin{align*}
b_n(z) = K_\Omega(z,y_n)-K_\Omega(y_n,z_0).
\end{align*}
Since $b_n(z_0)=0$ and each $b_n$ is 1-Lipschitz (with respect to the Kobayashi distance) we can pass to a subsequence such that $b_n \rightarrow b$ locally uniformly. Then
\begin{align*}
b(\gamma^{-1}z) = \lim_{n \rightarrow \infty} K_\Omega(\gamma^{-1} z,y_n)-K_\Omega(y_n,z_0) =  \lim_{n \rightarrow \infty} K_\Omega(z,\gamma y_n)-K_\Omega(y_n,z_0) =b(z)
\end{align*}
since
\begin{align*}
\abs{ K_\Omega(z,\gamma y_n)-K_\Omega(z, y_n)} \leq K_\Omega(y_n, \gamma y_n) \rightarrow 0.
\end{align*}
So 
\begin{align*}
b(\gamma^{-n} z_0) = b(z_0)=0
\end{align*}
for all $n \in \Nb$. 

\begin{observation} For any $t \in \Rb$
\begin{align*}
\overline{b^{-1}\Big((-\infty, t]\Big)}^{\Euc} \cap \partial \Omega \subset T_{\xi}^{\Cb} \partial \Omega.
\end{align*}
\end{observation}

\begin{proof} Suppose $w_m \in b^{-1}\Big((-\infty, t]\Big)$ and $w_m \rightarrow \eta \in \partial \Omega$. If $\eta \notin T_\xi^{\Cb} \partial \Omega$, then Theorem~\ref{thm:GP_convex} implies that there exists $R >0$ such that 
\begin{align*}
K_\Omega(w_m, z_0) &+ K_\Omega(z_0, y_n) - K_\Omega(w_m, y_n) \leq R.
\end{align*}
Then 
\begin{align*}
b(w_m) = \lim_{n \rightarrow \infty} K_\Omega(w_m, y_n)-K_\Omega(z_0, y_n) \geq K_\Omega(w_m, z_0)-R.
\end{align*}
However $K_\Omega(w_m, z_0) \rightarrow \infty$ since $K_\Omega$ is a proper metric on $\Omega$. So we have a contradiction.
\end{proof}

Using the previous observation, there exists $n_k \rightarrow \infty$ such that
\begin{align*}
\lim_{k \rightarrow \infty} d_{\Euc}\left( \gamma^{-n_k} z_0, T_{\xi}^{\Cb} \partial \Omega \right)=0.
\end{align*}
So $\Lc(\Omega) \cap T_{\xi}^{\Cb} \partial \Omega \neq \emptyset$. 

\end{proof}

\bibliographystyle{alpha}
\bibliography{complex_kob}

\begin{thebibliography}{{Zim}17a}

\bibitem[Aba89]{A1989}
Marco Abate.
\newblock {\em Iteration theory of holomorphic maps on taut manifolds}.
\newblock Research and Lecture Notes in Mathematics. Complex Analysis and
  Geometry. Mediterranean Press, Rende, 1989.

\bibitem[Ber60]{B1960}
Lipman Bers.
\newblock Spaces of {R}iemann surfaces as bounded domains.
\newblock {\em Bull. Amer. Math. Soc.}, 66:98--103, 1960.

\bibitem[BoP98]{BP1998b}
Zbigniew B\l~ocki and Peter Pflug.
\newblock Hyperconvexity and {B}ergman completeness.
\newblock {\em Nagoya Math. J.}, 151:221--225, 1998.

\bibitem[CY80]{CY1980}
Shiu~Yuen Cheng and Shing~Tung Yau.
\newblock On the existence of a complete {K}\"ahler metric on noncompact
  complex manifolds and the regularity of {F}efferman's equation.
\newblock {\em Comm. Pure Appl. Math.}, 33(4):507--544, 1980.

\bibitem[DFsH84]{DFH1984}
K.~Diederich, J.~E. Forn\ae~ss, and G.~Herbort.
\newblock Boundary behavior of the {B}ergman metric.
\newblock In {\em Complex analysis of several variables ({M}adison, {W}is.,
  1982)}, volume~41 of {\em Proc. Sympos. Pure Math.}, pages 59--67. Amer.
  Math. Soc., Providence, RI, 1984.

\bibitem[Ebe96]{E1996}
Patrick~B. Eberlein.
\newblock {\em Geometry of nonpositively curved manifolds}.
\newblock Chicago Lectures in Mathematics. University of Chicago Press,
  Chicago, IL, 1996.

\bibitem[FR87]{FR1987}
Franc Forstneri{\v{c}} and Jean-Pierre Rosay.
\newblock Localization of the {K}obayashi metric and the boundary continuity of
  proper holomorphic mappings.
\newblock {\em Math. Ann.}, 279(2):239--252, 1987.

\bibitem[Fra91]{F1991}
Sidney Frankel.
\newblock Applications of affine geometry to geometric function theory in
  several complex variables. {I}. {C}onvergent rescalings and intrinsic
  quasi-isometric structure.
\newblock In {\em Several complex variables and complex geometry, {P}art 2
  ({S}anta {C}ruz, {CA}, 1989)}, volume~52 of {\em Proc. Sympos. Pure Math.},
  pages 183--208. Amer. Math. Soc., Providence, RI, 1991.

\bibitem[Gra90]{G1990}
Ian Graham.
\newblock Distortion theorems for holomorphic maps between convex domains in
  {${\bf C}^n$}.
\newblock {\em Complex Variables Theory Appl.}, 15(1):37--42, 1990.

\bibitem[Gra91]{G1991}
Ian Graham.
\newblock Sharp constants for the {K}oebe theorem and for estimates of
  intrinsic metrics on convex domains.
\newblock In {\em Several complex variables and complex geometry, {P}art 2
  ({S}anta {C}ruz, {CA}, 1989)}, volume~52 of {\em Proc. Sympos. Pure Math.},
  pages 233--238. Amer. Math. Soc., Providence, RI, 1991.

\bibitem[Gri71]{G1971}
Phillip~A. Griffiths.
\newblock Complex-analytic properties of certain {Z}ariski open sets on
  algebraic varieties.
\newblock {\em Ann. of Math. (2)}, 94:21--51, 1971.

\bibitem[GS17]{GS2017}
S.~{Gupta} and H.~{Seshadri}.
\newblock {On domains biholomorphic to Teichm{\"u}ller spaces}.
\newblock {\em ArXiv e-prints}, January 2017.

\bibitem[Hel01]{H2001}
Sigurdur Helgason.
\newblock {\em Differential geometry, {L}ie groups, and symmetric spaces},
  volume~34 of {\em Graduate Studies in Mathematics}.
\newblock American Mathematical Society, Providence, RI, 2001.
\newblock Corrected reprint of the 1978 original.

\bibitem[Her99]{H1999}
Gregor Herbort.
\newblock The {B}ergman metric on hyperconvex domains.
\newblock {\em Math. Z.}, 232(1):183--196, 1999.

\bibitem[Kar05]{K2005b}
Anders Karlsson.
\newblock On the dynamics of isometries.
\newblock {\em Geom. Topol.}, 9:2359--2394, 2005.

\bibitem[KY96]{KY1996}
Kang-Tae Kim and Jiye Yu.
\newblock Boundary behavior of the {B}ergman curvature in strictly pseudoconvex
  polyhedral domains.
\newblock {\em Pacific J. Math.}, 176(1):141--163, 1996.

\bibitem[KZ16]{KZ2016}
Kang-Tae Kim and Liyou Zhang.
\newblock On the uniform squeezing property of bounded convex domains in
  {$\Bbb{C}^n$}.
\newblock {\em Pacific J. Math.}, 282(2):341--358, 2016.

\bibitem[LSY04]{LSY2004}
Kefeng Liu, Xiaofeng Sun, and Shing-Tung Yau.
\newblock Canonical metrics on the moduli space of riemann surfaces i.
\newblock {\em J. Differential Geom.}, 68(3):571--637, 11 2004.

\bibitem[LW18]{LW2018}
K.~{Liu} and Y.~{Wu}.
\newblock {Geometry of complex bounded domains with finite-volume quotients}.
\newblock {\em ArXiv e-prints}, January 2018.

\bibitem[Mar17]{M2017}
Vladimir Markovic.
\newblock {Carath{\'e}odory's Metrics on Teichm{\"u}ller Spaces and L-shaped
  pillowcases}.
\newblock To appear in \emph{Duke Math. J.}, 2017.

\bibitem[MT92]{MT1992}
Ngaiming Mok and I-Hsun Tsai.
\newblock Rigidity of convex realizations of irreducible bounded symmetric
  domains of rank $\geq 2$.
\newblock {\em Journal f{\"u}r die reine und angewandte Mathematik},
  431:91--122, 1992.

\bibitem[MY83]{MY1983}
Ngaiming Mok and Shing-Tung Yau.
\newblock Completeness of the {K}\"ahler-{E}instein metric on bounded domains
  and the characterization of domains of holomorphy by curvature conditions.
\newblock In {\em The mathematical heritage of {H}enri {P}oincar\'e, {P}art 1
  ({B}loomington, {I}nd., 1980)}, volume~39 of {\em Proc. Sympos. Pure Math.},
  pages 41--59. Amer. Math. Soc., Providence, RI, 1983.

\bibitem[NA17]{NA2017}
N.~Nikolov and L.~Andreev.
\newblock Boundary behavior of the squeezing functions of {$\Bbb{C}$}-convex
  domains and plane domains.
\newblock {\em Internat. J. Math.}, 28(5):1750031, 5, 2017.

\bibitem[Ohs81]{O1981}
Takeo Ohsawa.
\newblock A remark on the completeness of the {B}ergman metric.
\newblock {\em Proc. Japan Acad. Ser. A Math. Sci.}, 57(4):238--240, 1981.

\bibitem[Ros79]{R1979}
Jean-Pierre Rosay.
\newblock Sur une caract\'erisation de la boule parmi les domaines de {${\bf
  C}^{n}$} par son groupe d'automorphismes.
\newblock {\em Ann. Inst. Fourier (Grenoble)}, 29(4):ix, 91--97, 1979.

\bibitem[Won77]{W1977}
B.~Wong.
\newblock Characterization of the unit ball in {${\bf C}^{n}$} by its
  automorphism group.
\newblock {\em Invent. Math.}, 41(3):253--257, 1977.

\bibitem[Yau11]{Y2011}
Shing-Tung Yau.
\newblock A survey of geometric structure in geometric analysis.
\newblock In {\em Surveys in differential geometry. {V}olume {XVI}. {G}eometry
  of special holonomy and related topics}, volume~16 of {\em Surv. Differ.
  Geom.}, pages 325--347. Int. Press, Somerville, MA, 2011.

\bibitem[Yeu09]{Y2009}
Sai-Kee Yeung.
\newblock Geometry of domains with the uniform squeezing property.
\newblock {\em Adv. Math.}, 221(2):547--569, 2009.

\bibitem[{Zim}17a]{Z2017convex}
A.~{Zimmer}.
\newblock {The automorphism group and limit set of a bounded domain II: the
  convex case}.
\newblock {\em ArXiv e-prints}, December 2017.

\bibitem[Zim17b]{Z2017}
Andrew~M. Zimmer.
\newblock Characterizing domains by the limit set of their automorphism group.
\newblock {\em Adv. Math.}, 308:438--482, 2017.

\end{thebibliography}

\end{document}